\documentclass{amsart}
\usepackage{todonotes}
\usepackage{mathpazo}
\usepackage{enumerate}
\usepackage{subcaption}
\usepackage{ifthen}
\usepackage{tikz,pgf,pgfmath}
	\usetikzlibrary{shapes,calc,decorations.pathmorphing}
	\pgfdeclarelayer{background}
	\pgfdeclarelayer{middle}
	\pgfsetlayers{background,middle,main}
	\tikzstyle{edge}=[line width=.75pt]
	\tikzstyle{fnode}=[fill=black,draw=black,circle,scale=\s]
	\tikzstyle{pathnode}=[inner sep=.9pt]
	\tikzstyle{wordnode}=[fill=white,scale=.8,inner sep=.5pt]

\usepackage[pagebackref,
    colorlinks=true,
	urlcolor=blue!20!black,
	citecolor=green!50!black]{hyperref}
\usepackage{amssymb}
\usepackage[nobysame]{amsrefs}
\newtheorem{theorem}{Theorem}[section]
\newtheorem{proposition}[theorem]{Proposition}
\newtheorem{conjecture}[theorem]{Conjecture}
\newtheorem{lemma}[theorem]{Lemma}
\newtheorem{corollary}[theorem]{Corollary}
\theoremstyle{remark}
\newtheorem{definition}[theorem]{Definition}
\newtheorem{example}[theorem]{Example}
\newtheorem{remark}[theorem]{Remark}

\newcommand{\defn}[1]{{\color{green!50!black}\emph{#1}}}
\newcommand{\defs}{\stackrel{\mathsf{def}}{=}}
\newcommand{\ie}{\text{i.e.,}\;}

\renewcommand{\th}{^{\mathsf{th}}}
\newcommand{\two}{\mathbf{2}}
\newcommand{\Poset}{\mathbf{P}}
\newcommand{\JI}{\mathsf{J}}
\newcommand{\MI}{\mathsf{M}}
\newcommand{\Words}{\mathsf{W}}
\newcommand{\WordPoset}{\mathbf{W}}
\newcommand{\comporder}{\leq_{\mathsf{comp}}}
\newcommand{\compless}{<_{\mathsf{comp}}}
\newcommand{\compcover}{\lessdot_{\mathsf{comp}}}
\newcommand{\ab}{\mathfrak{a}}
\newcommand{\bb}{\mathfrak{b}}
\newcommand{\jb}{\mathfrak{j}}
\newcommand{\ob}{\mathfrak{o}}
\newcommand{\ub}{\mathfrak{u}}
\newcommand{\vb}{\mathfrak{v}}
\newcommand{\wb}{\mathfrak{w}}

\newcommand{\Supp}{\mathsf{Supp}}
\newcommand{\indeg}{\mathsf{in}}
\newcommand{\Can}{\mathsf{CJ}}
\newcommand{\htriangle}{\mathsf{H}}
\renewcommand{\min}{\mathsf{min}}
\renewcommand{\max}{\mathsf{max}}
\newcommand{\tops}{\mathsf{top}}
\newcommand{\pos}{\mathsf{atom}}
\makeatletter
\let\orgdescriptionlabel\descriptionlabel
\renewcommand*{\descriptionlabel}[1]{%
  \let\orglabel\label
  \let\label\@gobble
  \phantomsection
  \protected@edef\@currentlabel{#1\unskip}%
  \let\label\orglabel
  \orgdescriptionlabel{(#1)}%
}
\makeatother
\title{Combinatorics of $(m,n)$-Word Lattices}
\author{Henri M{\"u}hle}
\address{Qoniac GmbH, Dr.-K{\"u}lz-Ring 15, 01067 Dresden, Germany}
\email{henri.muehle@proton.me}

\begin{document}

\begin{abstract}
    We study the $(m,n)$-word lattices recently introduced by V.~Pilaud and D.~Poliakova in their study of generalized Hochschild polytopes. We prove that these lattices are extremal and constructable by interval doublings. Moreover, we describe further combinatorial properties of these lattices, such as their cardinality, their canonical join representations and their Galois graphs. 
\end{abstract}

\maketitle

\section{Introduction}

The Hochschild polytope was introduced in algebraic topology~\cite{rivera18combinatorial,saneblidze09bitwisted,saneblidze11homology} and has recently gained quite some interest due to its wealth of combinatorial properties. Its $1$-skeleton can be realized in terms of a particular order on Dyck paths~\cite{chapoton20some} and as the componentwise order of certain integer tuples~\cite{combe21geometric}. Intriguingly, the resulting poset, dubbed the \defn{Hochschild lattice}, is in fact a trim and interval-constructable lattice~\cite{combe21geometric,muehle22hochschild}. 

In \cite{muehle22hochschild}, a connection between the Hochschild lattice and a particular shuffle lattice from \cite{greene88posets} was established which nicely parallels the connection between the Tamari lattice and the lattice of noncrossing partitions through the so-called \defn{core label order}~\cite{reading11noncrossing}. Question~7.1 in \cite{muehle22hochschild} asks for a two-parameter generalization of the Hochschild lattice that extends the mentioned connection between Hochschild and shuffle lattices. One attempt in answering this question was made by the author together with T.~McConville by introducing \defn{bubble lattices}, which are an order extension of the shuffle lattices~\cite{mcconville22bubbleI,mcconville22bubbleII}. 

Another intriguing generalization of the Hochschild lattice was recently presented in \cite{pilaud23hochschild} from a geometric point of view. In that article, the story of Hochschild polytopes and Hochschild lattices was beautifully spun further, by introducing a new family of polytopes, the \defn{$(m,n)$-Hochschild polytopes} that arise as \defn{shadows} of \defn{$(m,n)$-multiplihedra}. We refer the reader to the excellent article~\cite{pilaud23hochschild} for the whole story. 

It turns out that the $1$-skeleton of the $(m,n)$-Hochschild polytope can be oriented in such a way that one obtains a lattice; the \defn{$m$-lighted $n$-shade (right) rotation lattice}, which simultaneously generalizes the boolean lattice (for $m=0$) and the original Hochschild lattice (for $m=1$). There is some computational evidence that this lattice is constructable by interval doublings, but in contrast to the Hochschild lattice, it is in general not extremal anymore~\cite[Remark~29]{pilaud23hochschild}.

However, the $m$-lighted $n$-shade (right) rotation lattice contains a natural quotient lattice, the \defn{$(m,n)$-word lattice}, which also generalizes the Hochschild lattice and seems to retain all its beautiful lattice-theoretic properties. The main purpose of this article is to confirm this suspicion and provide some lattice-theoretic and combinatorial properties of these $(m,n)$-word lattices.

\begin{theorem}\label{thm:main_theorem}
    For integers $m,n\geq 0$, the $(m,n)$-word lattice is extremal and constructable by interval doublings. 
\end{theorem}

The (as of yet) undefined lattice-theoretic concepts will be introduced in Section~\ref{sec:lattice_theory}. In Section~\ref{sec:mn_words}, we formally introduce the $(m,n)$-word lattices, and prove Theorem~\ref{thm:main_theorem} in Section~\ref{sec:mn_lattice_properties}. Finally, in Section~\ref{sec:mn_combinatorics} we investigate further combinatorial properties of $\WordPoset(m,n)$, such as its cardinality, its canonical join representations and its Galois graph.

\section{Lattice-theoretic Background}
    \label{sec:lattice_theory}

\subsection{Basics}

A \defn{poset} (short for \emph{partially ordered set}) is a finite set (the \defn{ground set}) equipped with a reflexive, antisymmetric and transitive relation (the \defn{order relation}). Usually, we denote the ground set by $P$, the order relation by $\leq$, and denote the corresponding poset $(P,\leq)$ by $\Poset$. The \defn{dual} poset of $\Poset$ is obtained by reversing the relation, \ie if $\Poset=(P,\leq)$, then its dual is $(P,\geq)$. 

A classical example of a poset is a \defn{chain}, \ie a poset where for any two elements $p,q\in P$ it holds that $p\leq q$ or $q\leq p$. If the ground set of a chain has $n$ elements, then we call it an \defn{$n$-chain}, and such a chain is isomorphic to $\mathbf{n}\defs\bigl([n],\leq\bigr)$, where $[n]=\{1,2,\ldots,n\}$.

The \defn{length} of $\Poset$, denoted by $\ell(\Poset)$, is one less than the maximum cardinality of a chain contained in $\Poset$. 

If $\Poset_1=(P_1,\leq_1)$ and $\Poset_2=(P_2,\leq_2)$ are two posets, then their \defn{direct product} is the poset $\Poset_1 \times \Poset_2 \defs (P_1\times P_2,\leq)$, where $(p_1,p_2) \leq (q_1,q_2)$ if and only if $p_1 \leq_1 q_1$ and $p_2 \leq_2 q_2$.

A \defn{lattice} is a poset in which every two elements $p,q\in P$ have a \defn{join} (\ie a unique minimal element that is above both $p$ and $q$) and a \defn{meet} (\ie a unique maximal element that is below both $p$ and $q$). In such a case we write $p\vee q$ for the join of $p$ and $q$ and $p\wedge q$ for the meet of $p$ and $q$.

\subsection{Interval-constructable Lattices}

Let $X\subseteq P$. The \defn{ideal} generated by $X$ is the set
\begin{displaymath}
    P_{\leq X} \defs \{p\in P\colon p\leq x\;\text{for some}\;x\in X\}.
\end{displaymath}

An \defn{interval} is a set $X\subseteq P$ with the property that there exist two elements $p,q\in P$ such that $X=\{x\colon p\leq x\leq q\}$. In this case, we usually write $[p,q]$ instead of $X$. The \defn{doubling} of $\Poset$ by $X$ is the subposet of $\Poset\times\two$ induced by the ground set
\begin{displaymath}
    P[X] \defs \Bigl(P_{\leq X} \times \{1\}\Bigr) \cup \Bigl(\bigl((P\setminus P_{\leq X})\cup X\bigr) \times \{2\}\Bigr).
\end{displaymath}

\begin{proposition}[\cite{day92doubling}]
    If $\Poset=(P,\leq)$ is a lattice and $X\subseteq P$ is an interval, then the doubling of $\Poset$ by $X$ is a lattice.
\end{proposition}

A lattice is \defn{interval constructable} if it can be obtained from the $2$-chain $\two$ by a sequence of interval doublings.

\subsection{Semidistributive lattices}

From now on, let $\Poset=(P,\leq)$ be a finite lattice. An element $j\in P$ is \defn{join irreducible} if it cannot be written as the join of two distinct elements. In other words, whenever $j=p\vee q$ for $p,q\in P$, then either $p=j$ or $q=j$. We denote the set of all join-irreducible elements of $\Poset$ by $\JI(\Poset)$. 

A \defn{join representation} of $p\in P$ is a set $A\subseteq P$ such that $p=\bigvee A$. If the set $\{P_{\leq A}\colon A\;\text{is a join representation of}\;p\}$ has a unique minimal element with respect to inclusion, then the join representation corresponding to this minimal element is the \defn{canonical join representation} of $p$, denoted by $\Can(p)$. 

A lattice is \defn{join semidistributive} if and only if every element has a canonical join representation~\cite[Theorem~2.24]{freese95free}. 

By duality, we can also define meet-irreducible elements, (canonical) meet representations and meet-semidistributive lattices. The set of meet-irreducible elements of $\Poset$ is denoted by $\MI(\Poset)$.

A lattice is \defn{semidistributive} if it is both join- and meet-semidistributive. 

\begin{lemma}[{\cite[Corollary~2.55]{freese95free}}]\label{lem:semidistributive_irreducibles}
    If $\Poset$ is semidistributive, then $\bigl\lvert\JI(\Poset)\bigr\rvert=\bigl\lvert\MI(\Poset)\bigr\rvert$.
\end{lemma}

\begin{theorem}[{\cite[Lemma~4.2]{day79characterizations}}]\label{thm:semidistributive_interval_constructible}
    Every interval-constructable lattice is semidistributive.
\end{theorem}

\subsection{Extremal lattices}

Following \cite{markowsky92primes}, the lattice $\Poset$ is \defn{extremal} if 
\begin{displaymath}
    \bigl\lvert\JI(\Poset)\bigr\rvert=\ell(\Poset)=\bigl\lvert\MI(\Poset)\bigr\rvert. 
\end{displaymath}

An element $x\in P$ is \defn{left modular} if $(p\vee x)\wedge q = p\vee(x\wedge q)$ holds for all $p,q\in P$ with $p<q$. If $\Poset$ is extremal and contains an chain of length $\ell(\Poset)$ consisting entirely of left-modular elements, then $\Poset$ is \defn{trim}~\cite{thomas07analogue}.  Surprisingly, the next result states that when a lattice is extremal and semidistributive it must necessarily be trim.

\begin{theorem}[{\cite[Theorem~1.4]{thomas19rowmotion}}]\label{thm:semidistributive_extremal_trim}
    Every extremal, semidistributive lattice is trim.
\end{theorem}

\section{Basics on $(m,n)$-Words}
    \label{sec:mn_words}

Let us start right away with the central definition of this article. 

\begin{definition}[{\cite[Definition~75]{pilaud23hochschild}}]
	Let $m,n\geq 0$. An \defn{$(m,n)$-word} is a word $w_{1}w_{2}\ldots w_{n}$ of length $n$ on the alphabet $\{0,1,\ldots,m{+}1\}$ such that
	\begin{description}
		\item[MN1\label{it:mn_1}] $w_{1}\neq m+1$,
		\item[MN2\label{it:mn_2}] for $1\leq s\leq m$, $w_{i}=s$ implies $w_{j}\geq s$ for all $j<i$.
	\end{description}
\end{definition}

As a convention, we denote $(m,n)$-words by a latin letter in fraktur font, e.g., $\wb$, and then denote its $i\th$ member by the same letter in regular font with subscript $i$, e.g., $w_i$. More precisely, when we deal with an $(m,n)$-word $\wb$, then we will occasionally access its letters through the variable $w_i$ without explicitly stating something along the lines of ``Let $\wb=w_1w_2\ldots w_n$.''.

Let $\Words(m,n)$ denote the set of $(m,n)$-words.  For two words $\ub=u_{1}u_{2}\ldots u_{n}$ and $\vb=v_{1}v_{2}\ldots v_{n}$ we write $\ub\comporder\vb$ if and only if $u_{i}\leq v_{i}$ for all $i\in[n]$. The \defn{$(m,n)$-word poset} is the poset $\WordPoset(m,n)\defs\bigl(\Words(m,n),\comporder\bigr)$.  It is quickly verified that $\WordPoset(m,n)$ has a unique minimal element, namely $\ob\defs 00\ldots0$, and a unique maximal element $m(m+1)(m+1)\ldots(m+1)$. 

\begin{theorem}[{\cite{pilaud23hochschild}*{Corollary~82}}]\label{thm:mn_lattice}
	For $m,n\geq 0$, the poset $\WordPoset(m,n)$ is a lattice.
\end{theorem}

The main purpose of this article is a thorough study of the lattice $\WordPoset(m,n)$ through which we will exhibit various nice properties. Before we can obtain certain enumerative results, it will be of major help to understand some of the structural properties of $\WordPoset(m,n)$.

\section{Lattice-theoretic Properties of $\WordPoset(m,n)$}
    \label{sec:mn_lattice_properties}

\subsection{Interval-constructability}

Given a word $\wb=w_{1}w_{2}\ldots w_{n}\in\WordPoset(m,n)$, we write $\min(\wb)$ for the minimum letter of $\wb$. For $i\in[0,m]$ let us define
\begin{equation}\label{eq:mn_slice}
	\Words^{(i)}(m,n) \defs \bigl\{\wb\in\Words(m,n)\colon \min(\wb)=i\bigr\}.
\end{equation}
Then, clearly, $\Words^{(0)}(m,n)=\Words(m,n)$ and $\Words^{(i-1)}(m,n)\supsetneq\Words^{(i)}(m,n)$ for $i\in[m]$.

\begin{lemma}\label{lem:slice_interval}
	Let $i\in[0,m]$. The poset $\bigl(\Words^{(i)}(m,n),\comporder\bigr)$ is an interval of $\WordPoset(m,n)$.
\end{lemma}
\begin{proof}
	Let $i\in[0,m]$ and consider the greatest element $\vb=v_{1}v_{2}\ldots v_{n}$ of $\WordPoset(m,n)$. This means that $v_{1}=m$ and $v_{2}=\cdots=v_{n}=m+1$.  Then, clearly $\min(\vb)=m\geq i$ which implies that $\vb\in\Words^{(i)}(m,n)$. Moreover, if $\ub=u_{1}u_{2}\cdots u_{n}$ with $u_{1}=u_{2}=\cdots=u_{n}=i$, then $\ub\in\Words^{(i)}(m,n)$, too. Clearly, any $\wb\in\Words^{(i)}$ must have $\ub\comporder\wb\comporder\vb$, which implies that the induced subposet $\bigl(\Words^{(i)}(m,n),\comporder\bigr)$ is contained in the interval $[\ub,\vb]$ of $\WordPoset(m,n)$.
	
	Conversely, let $\wb\in\Words(m,n)$ such that $\ub\comporder\wb\comporder\vb$. This means, however, that $\min(\wb)\geq i$ and thus $\wb\in\Words^{(i)}(m,n)$. This proves the claim.
\end{proof}

If $\wb$ is a word, then for any integer $i$, we write $\wb i$ for the word obtained by adding the letter $i$ to the end of $\wb$.

\begin{lemma}\label{lem:slice_enrichment}
	Let $i\in[0,m]$. For $\wb\in\Words^{(i)}(m,n)$, we have $\wb i\in\Words(m,n+1)$.
\end{lemma}
\begin{proof}
	Let $\wb=\Words(m,n)$. Since $\wb\in\Words^{(i)}(m,n)\subseteq\Words(m,n)$, it follows that $w_{i}\in[0,m+1]$ and $w_{1}\neq m+1$ by \eqref{it:mn_1}. Moreover, for any $i\in[n]$ it is guaranteed by \eqref{it:mn_2} that $w_{i}=s\in[1,m]$ implies $w_{j}\geq s$ for all $j<i$. 
	
	Then, the word $\wb i$ is formed of letters from the interval $[0,m+1]$, too, and satisfies \eqref{it:mn_1}.  Moreover, $\wb\in\Words^{(i)}(m,n)$ implies that $w_{j}\geq i$ for all $j\in[n]$, so that $\wb i$ also satisfies \eqref{it:mn_2}. It follows that $\wb i\in\Words(m,n+1)$.
\end{proof}

We will now prove one part of our main result.

\begin{theorem}\label{thm:mn_doubling}
	For $m,n\geq 0$, the lattice $\WordPoset(m,n)$ is interval-constructable.
\end{theorem}
\begin{proof}
	We fix $m\geq 0$ and proceed by induction on $n$. If $n=0$, then $\WordPoset(m,0)$ is the singleton lattice consisting of the empty word. If $n=1$, then $\WordPoset(m,1)$ is an $m{+}1$-chain, consisting of the words $0,1,\ldots,m$. Both of these lattices are interval-constructable.

	Now let $n\geq 1$ and assume inductively that $\WordPoset(m,n)$ is interval-constructable.  We now describe how to obtain $\WordPoset(m,n+1)$ by $m+1$ successive interval doublings.

	First, we consider the poset $\Poset^{(0)}\defs\bigl(P^{(0)},\comporder\bigr)$, where
	\begin{displaymath}
		P^{(0)} \defs \bigl\{\wb0\colon \wb\in\Words(m,n)\bigr\}
			\uplus \bigl\{\wb (m{+}1)\colon \wb\in\Words(m,n)\bigr\}.
	\end{displaymath}
	Then, clearly, $\Poset^{(0)}\cong\WordPoset(m,n)\times\two$. Moreover, \eqref{it:mn_2} implies that $P^{(0)}\subseteq\Words(m,n{+}1)$.

	Now, for $i\in[m]$, we set $j \defs m+1-i$ and define
	\begin{displaymath}
		P^{(i)} \defs \bigl\{\wb j\colon \wb\in \Words^{(j)}(m,n)\bigr\}.
	\end{displaymath}
	Lemma~\ref{lem:slice_enrichment} implies that $P^{(i)}\subseteq\Words(m,n+1)$. Moreover, it follows from Lemma~\ref{lem:slice_interval} that $\Words^{(j)}(m,n)$ is an interval of $\WordPoset(m,n)$, which by construction is isomorphic to the interval of $\Poset^{(0)}$ induced by $I^{(j)}\defs\bigl\{\wb0 \colon \wb\in\Words^{(j)}(m,n)\bigr\}$.  Thus, if we define
	\begin{displaymath}
		\Poset^{(i)} \defs \bigl(P^{(0)}\uplus P^{(1)}\uplus\cdots\uplus P^{(i)},\comporder\bigr),
	\end{displaymath}
	then $\Poset^{(i)}$ is obtained from $\Poset^{(i-1)}$ through doubling by $I^{(j)}$.

	\medskip

	It remains to show that $\Poset^{(m)}$ is indeed equal to $\WordPoset(m,n+1)$. Since both posets have as ground sets words of length $n+1$ and use componentwise order, it is enough to show that the ground set of $\Poset^{(m)}$ equals $\Words(m,n+1)$.

	The ground set of $\Poset^{(m)}$ is $P^{(0)}\uplus P^{(1)}\uplus\cdots\uplus P^{(m)}$, and we have already established that $P^{(i)}\subseteq\Words(m,n+1)$ for all $0\leq i\leq m$.\\
	Conversely, let $\wb\in\Words(m,n+1)$, and suppose that $\wb=w_{1}w_{2}\ldots w_{n}w_{n+1}$. Then, by definition, $w_{1}w_{2}\ldots w_{n}\in\Words(m,n)$. If $w_{n+1}\in\{0,m+1\}$, then $\wb\in P^{(0)}$. Otherwise, it must be that $w_{n+1}=m+1-i$ for some $i\in[m]$, which implies $\wb\in P^{(i)}$. Thus, $\wb$ is in the ground set of $\Poset^{(m)}$ and the proof is complete.
\end{proof}

For $m=1$, the doubling construction described in the proof of Theorem~\ref{thm:mn_doubling} consists of two steps and agrees with the doubling procedure described in \cite{combe21geometric}*{Section~3.3}.  Our doubling procedure is illustrated in Figure~\ref{fig:23_doubling}.

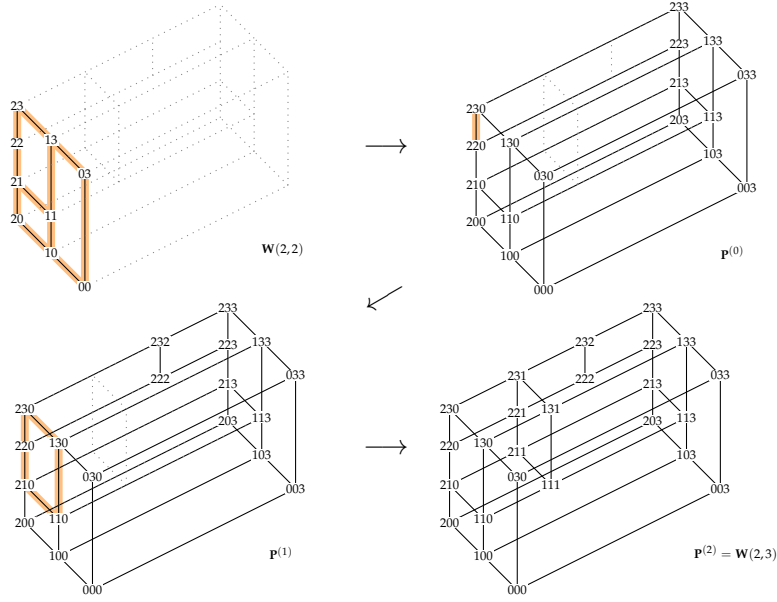
\begin{figure}
	\centering
\begin{tikzpicture}
	\def\x{.9};
	\def\y{.9};
	\draw(1,5) node{\begin{tikzpicture}[scale=.5]
		\coordinate(n000) at (3*\x,.9*\y);
		\coordinate(n003) at (9*\x,3.9*\y);
		\coordinate(n030) at (3*\x,4.25*\y);
		\coordinate(n033) at (9*\x,7.25*\y);
		\coordinate(n100) at (2*\x,1.9*\y);
		\coordinate(n103) at (8*\x,4.9*\y);
		\coordinate(n110) at (2*\x,3*\y);
		\coordinate(n111) at (4*\x,4*\y);
		\coordinate(n113) at (8*\x,6*\y);
		\coordinate(n130) at (2*\x,5.25*\y);
		\coordinate(n131) at (4*\x,6.25*\y);
		\coordinate(n133) at (8*\x,8.25*\y);
		\coordinate(n200) at (1*\x,2.9*\y);
		\coordinate(n203) at (7*\x,5.9*\y);
		\coordinate(n210) at (1*\x,4*\y);
		\coordinate(n211) at (3*\x,5*\y);
		\coordinate(n213) at (7*\x,7*\y);
		\coordinate(n220) at (1*\x,5.15*\y);
		\coordinate(n221) at (3*\x,6.15*\y);
		\coordinate(n222) at (5*\x,7.15*\y);
		\coordinate(n223) at (7*\x,8.15*\y);
		\coordinate(n230) at (1*\x,6.25*\y);
		\coordinate(n231) at (3*\x,7.25*\y);
		\coordinate(n232) at (5*\x,8.25*\y);
		\coordinate(n233) at (7*\x,9.25*\y);
		\draw[orange!50!white,line width=3pt](n000) -- (n100);
		\draw[orange!50!white,line width=3pt](n000) -- (n030);
		\draw[orange!50!white,line width=3pt](n030) -- (n130);
		\draw[orange!50!white,line width=3pt](n100) -- (n110);
		\draw[orange!50!white,line width=3pt](n100) -- (n200);
		\draw[orange!50!white,line width=3pt](n110) -- (n130);
		\draw[orange!50!white,line width=3pt](n110) -- (n210);
		\draw[orange!50!white,line width=3pt](n130) -- (n230);
		\draw[orange!50!white,line width=3pt](n200) -- (n210);
		\draw[orange!50!white,line width=3pt](n210) -- (n220);
		\draw[orange!50!white,line width=3pt](n220) -- (n230);
		\draw(n000) -- (n100);
		\draw(n000) -- (n030);
		\draw(n030) -- (n130);
		\draw(n100) -- (n110);
		\draw(n100) -- (n200);
		\draw(n110) -- (n130);
		\draw(n110) -- (n210);
		\draw(n130) -- (n230);
		\draw(n200) -- (n210);
		\draw(n210) -- (n220);
		\draw(n220) -- (n230);
		\draw[dotted,thin,gray](n000) -- (n003);
		\draw[dotted,thin,gray](n003) -- (n033);
		\draw[dotted,thin,gray](n003) -- (n103);
		\draw[dotted,thin,gray](n030) -- (n033);
		\draw[dotted,thin,gray](n033) -- (n133);
		\draw[dotted,thin,gray](n100) -- (n103);
		\draw[dotted,thin,gray](n103) -- (n113);
		\draw[dotted,thin,gray](n103) -- (n203);
		\draw[dotted,thin,gray](n110) -- (n111);
		\draw[dotted,thin,gray](n111) -- (n113);
		\draw[dotted,thin,gray](n111) -- (n131);
		\draw[dotted,thin,gray](n111) -- (n211);
		\draw[dotted,thin,gray](n113) -- (n133);
		\draw[dotted,thin,gray](n113) -- (n213);
		\draw[dotted,thin,gray](n130) -- (n131);
		\draw[dotted,thin,gray](n131) -- (n133);
		\draw[dotted,thin,gray](n131) -- (n231);
		\draw[dotted,thin,gray](n133) -- (n233);
		\draw[dotted,thin,gray](n200) -- (n203);
		\draw[dotted,thin,gray](n203) -- (n213);
		\draw[dotted,thin,gray](n210) -- (n211);
		\draw[dotted,thin,gray](n211) -- (n213);
		\draw[dotted,thin,gray](n211) -- (n221);
		\draw[dotted,thin,gray](n213) -- (n223);
		\draw[dotted,thin,gray](n220) -- (n221);
		\draw[dotted,thin,gray](n230) -- (n231);
		\draw[dotted,thin,gray](n221) -- (n231);
		\draw[dotted,thin,gray](n221) -- (n222);
		\draw[dotted,thin,gray](n222) -- (n223);
		\draw[dotted,thin,gray](n222) -- (n232);
		\draw[dotted,thin,gray](n223) -- (n233);
		\draw[dotted,thin,gray](n231) -- (n232);
		\draw[dotted,thin,gray](n232) -- (n233);
		\draw(n000) node[wordnode]{\tiny $00$};
		\draw(n030) node[wordnode]{\tiny $03$};
		\draw(n100) node[wordnode]{\tiny $10$};
		\draw(n110) node[wordnode]{\tiny $11$};
		\draw(n130) node[wordnode]{\tiny $13$};
		\draw(n200) node[wordnode]{\tiny $20$};
		\draw(n210) node[wordnode]{\tiny $21$};
		\draw(n220) node[wordnode]{\tiny $22$};
		\draw(n230) node[wordnode]{\tiny $23$};
		\draw(8*\x,2*\y) node[anchor=west,scale=.8]{\tiny $\WordPoset(2,2)$};
	\end{tikzpicture}};
	\draw(4,5) node{$\longrightarrow$};
	\draw(7,5) node{\begin{tikzpicture}[scale=.5]
		\coordinate(n000) at (3*\x,.9*\y);
		\coordinate(n003) at (9*\x,3.9*\y);
		\coordinate(n030) at (3*\x,4.25*\y);
		\coordinate(n033) at (9*\x,7.25*\y);
		\coordinate(n100) at (2*\x,1.9*\y);
		\coordinate(n103) at (8*\x,4.9*\y);
		\coordinate(n110) at (2*\x,3*\y);
		\coordinate(n111) at (4*\x,4*\y);
		\coordinate(n113) at (8*\x,6*\y);
		\coordinate(n130) at (2*\x,5.25*\y);
		\coordinate(n131) at (4*\x,6.25*\y);
		\coordinate(n133) at (8*\x,8.25*\y);
		\coordinate(n200) at (1*\x,2.9*\y);
		\coordinate(n203) at (7*\x,5.9*\y);
		\coordinate(n210) at (1*\x,4*\y);
		\coordinate(n211) at (3*\x,5*\y);
		\coordinate(n213) at (7*\x,7*\y);
		\coordinate(n220) at (1*\x,5.15*\y);
		\coordinate(n221) at (3*\x,6.15*\y);
		\coordinate(n222) at (5*\x,7.15*\y);
		\coordinate(n223) at (7*\x,8.15*\y);
		\coordinate(n230) at (1*\x,6.25*\y);
		\coordinate(n231) at (3*\x,7.25*\y);
		\coordinate(n232) at (5*\x,8.25*\y);
		\coordinate(n233) at (7*\x,9.25*\y);
		\draw[orange!50!white,line width=3pt](n220) -- (n230);
		\draw(n000) -- (n003);
		\draw(n000) -- (n030);
		\draw(n000) -- (n100);
		\draw(n003) -- (n033);
		\draw(n003) -- (n103);
		\draw(n030) -- (n033);
		\draw(n030) -- (n130);
		\draw(n033) -- (n133);
		\draw(n100) -- (n103);
		\draw(n100) -- (n110);
		\draw(n100) -- (n200);
		\draw(n103) -- (n113);
		\draw(n103) -- (n203);
		\draw(n110) -- (n210);
		\draw(n110) -- (n130);
		\draw(n110) -- (n113);
		\draw(n113) -- (n133);
		\draw(n113) -- (n213);
		\draw(n130) -- (n133);
		\draw(n130) -- (n230);
		\draw(n133) -- (n233);
		\draw(n200) -- (n203);
		\draw(n200) -- (n210);
		\draw(n203) -- (n213);
		\draw(n210) -- (n213);
		\draw(n210) -- (n220);
		\draw(n213) -- (n223);
		\draw(n220) -- (n223);
		\draw(n220) -- (n230);
		\draw(n223) -- (n233);
		\draw(n230) -- (n233);
		\draw[dotted,thin,gray](n200) -- (n203);
		\draw[dotted,thin,gray](n110) -- (n111);
		\draw[dotted,thin,gray](n111) -- (n113);
		\draw[dotted,thin,gray](n111) -- (n131);
		\draw[dotted,thin,gray](n111) -- (n211);
		\draw[dotted,thin,gray](n130) -- (n131);
		\draw[dotted,thin,gray](n131) -- (n133);
		\draw[dotted,thin,gray](n131) -- (n231);
		\draw[dotted,thin,gray](n210) -- (n211);
		\draw[dotted,thin,gray](n211) -- (n213);
		\draw[dotted,thin,gray](n211) -- (n221);
		\draw[dotted,thin,gray](n220) -- (n221);
		\draw[dotted,thin,gray](n221) -- (n222);
		\draw[dotted,thin,gray](n221) -- (n231);
		\draw[dotted,thin,gray](n222) -- (n223);
		\draw[dotted,thin,gray](n222) -- (n232);
		\draw[dotted,thin,gray](n230) -- (n231);
		\draw[dotted,thin,gray](n231) -- (n232);
		\draw[dotted,thin,gray](n232) -- (n233);
		\draw(n000) node[wordnode]{\tiny $000$};
		\draw(n003) node[wordnode]{\tiny $003$};
		\draw(n030) node[wordnode]{\tiny $030$};
		\draw(n033) node[wordnode]{\tiny $033$};
		\draw(n100) node[wordnode]{\tiny $100$};
		\draw(n103) node[wordnode]{\tiny $103$};
		\draw(n110) node[wordnode]{\tiny $110$};
		\draw(n113) node[wordnode]{\tiny $113$};
		\draw(n130) node[wordnode]{\tiny $130$};
		\draw(n133) node[wordnode]{\tiny $133$};
		\draw(n200) node[wordnode]{\tiny $200$};
		\draw(n203) node[wordnode]{\tiny $203$};
		\draw(n210) node[wordnode]{\tiny $210$};
		\draw(n213) node[wordnode]{\tiny $213$};
		\draw(n220) node[wordnode]{\tiny $220$};
		\draw(n223) node[wordnode]{\tiny $223$};
		\draw(n230) node[wordnode]{\tiny $230$};
		\draw(n233) node[wordnode]{\tiny $233$};
		\draw(8*\x,2*\y) node[anchor=west,scale=.8]{\tiny $\Poset^{(0)}$};
	\end{tikzpicture}};
	\draw(4,3) node[rotate=30]{$\longleftarrow$};
	\draw(1,1) node{\begin{tikzpicture}[scale=.5]
		\coordinate(n000) at (3*\x,.9*\y);
		\coordinate(n003) at (9*\x,3.9*\y);
		\coordinate(n030) at (3*\x,4.25*\y);
		\coordinate(n033) at (9*\x,7.25*\y);
		\coordinate(n100) at (2*\x,1.9*\y);
		\coordinate(n103) at (8*\x,4.9*\y);
		\coordinate(n110) at (2*\x,3*\y);
		\coordinate(n111) at (4*\x,4*\y);
		\coordinate(n113) at (8*\x,6*\y);
		\coordinate(n130) at (2*\x,5.25*\y);
		\coordinate(n131) at (4*\x,6.25*\y);
		\coordinate(n133) at (8*\x,8.25*\y);
		\coordinate(n200) at (1*\x,2.9*\y);
		\coordinate(n203) at (7*\x,5.9*\y);
		\coordinate(n210) at (1*\x,4*\y);
		\coordinate(n211) at (3*\x,5*\y);
		\coordinate(n213) at (7*\x,7*\y);
		\coordinate(n220) at (1*\x,5.15*\y);
		\coordinate(n221) at (3*\x,6.15*\y);
		\coordinate(n222) at (5*\x,7.15*\y);
		\coordinate(n223) at (7*\x,8.15*\y);
		\coordinate(n230) at (1*\x,6.25*\y);
		\coordinate(n231) at (3*\x,7.25*\y);
		\coordinate(n232) at (5*\x,8.25*\y);
		\coordinate(n233) at (7*\x,9.25*\y);
		\draw[orange!50!white,line width=3pt](n110) -- (n210);
		\draw[orange!50!white,line width=3pt](n110) -- (n130);
		\draw[orange!50!white,line width=3pt](n130) -- (n230);
		\draw[orange!50!white,line width=3pt](n210) -- (n220);
		\draw[orange!50!white,line width=3pt](n220) -- (n230);
		\draw(n000) -- (n003);
		\draw(n000) -- (n030);
		\draw(n000) -- (n100);
		\draw(n003) -- (n033);
		\draw(n003) -- (n103);
		\draw(n030) -- (n033);
		\draw(n030) -- (n130);
		\draw(n033) -- (n133);
		\draw(n100) -- (n103);
		\draw(n100) -- (n110);
		\draw(n100) -- (n200);
		\draw(n103) -- (n113);
		\draw(n103) -- (n203);
		\draw(n110) -- (n210);
		\draw(n110) -- (n130);
		\draw(n110) -- (n113);
		\draw(n113) -- (n133);
		\draw(n113) -- (n213);
		\draw(n130) -- (n133);
		\draw(n130) -- (n230);
		\draw(n133) -- (n233);
		\draw(n200) -- (n203);
		\draw(n200) -- (n210);
		\draw(n203) -- (n213);
		\draw(n210) -- (n213);
		\draw(n210) -- (n220);
		\draw(n213) -- (n223);
		\draw(n220) -- (n222);
		\draw(n220) -- (n230);
		\draw(n222) -- (n232);
		\draw(n222) -- (n223);
		\draw(n223) -- (n233);
		\draw(n230) -- (n232);
		\draw(n232) -- (n233);
		\draw[dotted,thin,gray](n200) -- (n203);
		\draw[dotted,thin,gray](n110) -- (n111);
		\draw[dotted,thin,gray](n111) -- (n113);
		\draw[dotted,thin,gray](n111) -- (n131);
		\draw[dotted,thin,gray](n111) -- (n211);
		\draw[dotted,thin,gray](n130) -- (n131);
		\draw[dotted,thin,gray](n131) -- (n133);
		\draw[dotted,thin,gray](n131) -- (n231);
		\draw[dotted,thin,gray](n210) -- (n211);
		\draw[dotted,thin,gray](n211) -- (n213);
		\draw[dotted,thin,gray](n211) -- (n221);
		\draw[dotted,thin,gray](n220) -- (n221);
		\draw[dotted,thin,gray](n221) -- (n222);
		\draw[dotted,thin,gray](n221) -- (n231);
		\draw[dotted,thin,gray](n230) -- (n231);
		\draw[dotted,thin,gray](n231) -- (n232);
		\draw(n000) node[wordnode]{\tiny $000$};
		\draw(n003) node[wordnode]{\tiny $003$};
		\draw(n030) node[wordnode]{\tiny $030$};
		\draw(n033) node[wordnode]{\tiny $033$};
		\draw(n100) node[wordnode]{\tiny $100$};
		\draw(n103) node[wordnode]{\tiny $103$};
		\draw(n110) node[wordnode]{\tiny $110$};
		\draw(n113) node[wordnode]{\tiny $113$};
		\draw(n130) node[wordnode]{\tiny $130$};
		\draw(n133) node[wordnode]{\tiny $133$};
		\draw(n200) node[wordnode]{\tiny $200$};
		\draw(n203) node[wordnode]{\tiny $203$};
		\draw(n210) node[wordnode]{\tiny $210$};
		\draw(n213) node[wordnode]{\tiny $213$};
		\draw(n220) node[wordnode]{\tiny $220$};
		\draw(n222) node[wordnode]{\tiny $222$};
		\draw(n223) node[wordnode]{\tiny $223$};
		\draw(n230) node[wordnode]{\tiny $230$};
		\draw(n232) node[wordnode]{\tiny $232$};
		\draw(n233) node[wordnode]{\tiny $233$};
		\draw(8*\x,2*\y) node[anchor=west,scale=.8]{\tiny $\Poset^{(1)}$};
	\end{tikzpicture}};
	\draw(4,1) node{$\longrightarrow$};
	\draw(7,1) node{\begin{tikzpicture}[scale=.5]
		\coordinate(n000) at (3*\x,.9*\y);
		\coordinate(n003) at (9*\x,3.9*\y);
		\coordinate(n030) at (3*\x,4.25*\y);
		\coordinate(n033) at (9*\x,7.25*\y);
		\coordinate(n100) at (2*\x,1.9*\y);
		\coordinate(n103) at (8*\x,4.9*\y);
		\coordinate(n110) at (2*\x,3*\y);
		\coordinate(n111) at (4*\x,4*\y);
		\coordinate(n113) at (8*\x,6*\y);
		\coordinate(n130) at (2*\x,5.25*\y);
		\coordinate(n131) at (4*\x,6.25*\y);
		\coordinate(n133) at (8*\x,8.25*\y);
		\coordinate(n200) at (1*\x,2.9*\y);
		\coordinate(n203) at (7*\x,5.9*\y);
		\coordinate(n210) at (1*\x,4*\y);
		\coordinate(n211) at (3*\x,5*\y);
		\coordinate(n213) at (7*\x,7*\y);
		\coordinate(n220) at (1*\x,5.15*\y);
		\coordinate(n221) at (3*\x,6.15*\y);
		\coordinate(n222) at (5*\x,7.15*\y);
		\coordinate(n223) at (7*\x,8.15*\y);
		\coordinate(n230) at (1*\x,6.25*\y);
		\coordinate(n231) at (3*\x,7.25*\y);
		\coordinate(n232) at (5*\x,8.25*\y);
		\coordinate(n233) at (7*\x,9.25*\y);
		\draw(n000) -- (n003);
		\draw(n000) -- (n030);
		\draw(n000) -- (n100);
		\draw(n003) -- (n033);
		\draw(n003) -- (n103);
		\draw(n030) -- (n033);
		\draw(n030) -- (n130);
		\draw(n033) -- (n133);
		\draw(n100) -- (n103);
		\draw(n100) -- (n110);
		\draw(n100) -- (n200);
		\draw(n103) -- (n113);
		\draw(n103) -- (n203);
		\draw(n110) -- (n210);
		\draw(n110) -- (n111);
		\draw(n110) -- (n130);
		\draw(n111) -- (n113);
		\draw(n111) -- (n131);
		\draw(n111) -- (n211);
		\draw(n113) -- (n133);
		\draw(n113) -- (n213);
		\draw(n130) -- (n131);
		\draw(n130) -- (n230);
		\draw(n131) -- (n133);
		\draw(n131) -- (n231);
		\draw(n133) -- (n233);
		\draw(n200) -- (n203);
		\draw(n200) -- (n210);
		\draw(n203) -- (n213);
		\draw(n210) -- (n211);
		\draw(n210) -- (n220);
		\draw(n211) -- (n213);
		\draw(n211) -- (n221);
		\draw(n213) -- (n223);
		\draw(n220) -- (n221);
		\draw(n220) -- (n230);
		\draw(n221) -- (n222);
		\draw(n221) -- (n231);
		\draw(n222) -- (n232);
		\draw(n222) -- (n223);
		\draw(n223) -- (n233);
		\draw(n230) -- (n231);
		\draw(n231) -- (n232);
		\draw(n232) -- (n233);
		\draw(n000) node[wordnode]{\tiny $000$};
		\draw(n003) node[wordnode]{\tiny $003$};
		\draw(n030) node[wordnode]{\tiny $030$};
		\draw(n033) node[wordnode]{\tiny $033$};
		\draw(n100) node[wordnode]{\tiny $100$};
		\draw(n103) node[wordnode]{\tiny $103$};
		\draw(n110) node[wordnode]{\tiny $110$};
		\draw(n111) node[wordnode]{\tiny $111$};
		\draw(n113) node[wordnode]{\tiny $113$};
		\draw(n130) node[wordnode]{\tiny $130$};
		\draw(n131) node[wordnode]{\tiny $131$};
		\draw(n133) node[wordnode]{\tiny $133$};
		\draw(n200) node[wordnode]{\tiny $200$};
		\draw(n203) node[wordnode]{\tiny $203$};
		\draw(n210) node[wordnode]{\tiny $210$};
		\draw(n211) node[wordnode]{\tiny $211$};
		\draw(n213) node[wordnode]{\tiny $213$};
		\draw(n220) node[wordnode]{\tiny $220$};
		\draw(n221) node[wordnode]{\tiny $221$};
		\draw(n222) node[wordnode]{\tiny $222$};
		\draw(n223) node[wordnode]{\tiny $223$};
		\draw(n230) node[wordnode]{\tiny $230$};
		\draw(n231) node[wordnode]{\tiny $231$};
		\draw(n232) node[wordnode]{\tiny $232$};
		\draw(n233) node[wordnode]{\tiny $233$};
		\draw(8*\x,2*\y) node[anchor=west,scale=.8]{\tiny $\Poset^{(2)}=\WordPoset(2,3)$};
	\end{tikzpicture}};
\end{tikzpicture}
	\caption{The doubling procedure that produces $\WordPoset(2,3)$ from $\WordPoset(2,2)$. Along each arrow, the highlighted interval is doubled.}
	\label{fig:23_doubling}
\end{figure}

\begin{corollary}\label{cor:mn_semidistributive}
    For $m,n\geq 0$, the lattice $\WordPoset(m,n)$ is semidistributive.
\end{corollary}
\begin{proof}
    This follows from Theorems~\ref{thm:mn_doubling} and \ref{thm:semidistributive_interval_constructible}.
\end{proof}

%

\subsection{Extremality}\label{sec:extremality}

Let us consider the following words:
\begin{align*}
    \ab^{(i,j)} & \defs \begin{cases}a_k=j, & \text{if}\;k\leq i,\\a_k=0, & \text{if}\;k>i,\end{cases}\\
    \bb^{(i)} & \defs \begin{cases}b_k=m+1, & \text{if}\;k=i,\\b_k=0, & \text{if}\;k\neq i.\end{cases}
\end{align*}
By \eqref{it:mn_1}, $\ab^{(i,j)}\in\Words(m,n)$ if $i\in[n]$ and $j\in[0,m]$ and $\bb_{i}\in\Words(m,n)$ if $i\in[2,n]$.

\begin{lemma}\label{lem:mn_irreducibles}
    An $(m,n)$-word $\wb\in\Words(m,n)$ is join-irreducible in $\WordPoset(m,n)$ if and only if either $\wb=\ab^{(i,j)}$ for $i\in[n]$ and $j\in[m]$ or $\wb=\bb^{(i)}$ for $i\in[2,n]$.
\end{lemma}
\begin{proof}
    For all $i\in[n]$, $\ab^{(i,0)}=\ob$ is the bottom element of $\WordPoset(m,n)$ and therefore cannot be join-irreducible. 

    As $\bb^{(i)}$ has only one non-zero entry for all $i\in[2,n]$, it is clear that it is join-irreducible. In fact, \eqref{it:mn_2} implies $\ob\compcover\bb^{(i)}$.

    Likewise, if $i\in[n]$ and $j\in[m]$, then \eqref{it:mn_2} implies that only the $i\th$ letter of $\ab^{(i,j)}$ can be decreased so that one still obtains an $(m,n)$-word. This implies that $\ab^{(i,j)}$ covers a unique element, and is therefore join-irreducible.

    \medskip

    Conversely, suppose that $\vb=v_1v_2\ldots v_n$ is join-irreducible. This means that $\vb$ covers a unique element $\ub=u_1u_2\ldots u_n$. By construction, there exists a unique index $i\in[n]$ such that $u_i<v_i$ and $u_j=v_j$ for all $j\neq i$. This implies in particular that $v_i\neq 0$.

    (i) Assume that $v_i=m+1$. Suppose further that there exists $k\in[n]$, $k\neq i$, be the largest index such that $v_k>0$. Then, consider the word $\wb=w_1w_2\ldots w_n$ with $w_k=0$ and $w_j=v_j$ for $j\neq k$. Since $w_1=v_1<m+1$, \eqref{it:mn_1} is satisfied. If $w_\ell=s\notin\{0,m+1\}$, then $\ell<k$ and $v_\ell=s$ and by \eqref{it:mn_2}, $w_j=v_j>s$ for all $j<\ell$. This means that $\wb\in\Words(m,n)$ with $\wb\compless\vb$ but $\wb\not\comporder\ub$ which contradicts $\vb$ being join-irreducible. We conclude that when $v_i=m+1$, it must be the case that $\vb=\bb^{(i)}$.

    (ii) Assume that $v_i=s\notin\{0,m+1\}$. By \eqref{it:mn_2}, $v_j\geq s$ for all $j<i$. Suppose that there exists a largest index $k\in[i-1]$ such that $v_k>s$. Then, consider the word $\wb=w_1w_2\ldots w_n$ with $w_k=s$ and $w_j=v_j$ for all $j\neq k$. It is clear that $\wb\in\Words(m,n)$. Then $\wb\compless\vb$ and $\wb\not\comporder\ub$, which contradicts $\vb$ being join-irreducible. It follows that $v_1=v_2=\cdots=v_i=s$. Now suppose that there exists a largest index $k\in\{i+1,i+2,\ldots,n\}$ such that $v_k>0$. If we consider the word $\wb'=w'_1w'_2\ldots w'_n$ with $w'_k=0$ and $w'_j=v_j$ for all $j\neq k$, then analogously to (i) we obtain a contradiction to $\vb$ being join-irreducible. We conclude that when $v_i=s\notin\{0,m+1\}$, it must be the case that $\vb=\ab^{(i,s)}$.
\end{proof}

\begin{lemma}\label{lem:mn_long_chain}
    For $m\geq 0$ and $n>0$, $\WordPoset(m,n)$ contains a chain of length $(m+1)n-1$.
\end{lemma}
\begin{proof}
    We prove this statement by induction on $n$. If $n=1$, then $\WordPoset(m,1)$ is itself a chain of length $m$ (because it consists of the $m+1$ words $0,1,\ldots,m$).

    Now let $n>1$. By Theorem~\ref{thm:mn_doubling}, we may obtain $\WordPoset(m,n)$ from $\WordPoset(m,n-1)$ by consecutively doubling $m+1$ intervals. This implies that we may extend a chain of length $k$ from bottom to top in $\WordPoset(m,n-1)$ to a chain of length $k+m+1$ from bottom to top in $\WordPoset(m,n)$. 
    
    By induction assumption, we can find a chain of length $(m+1)(n-1)-1$ from bottom to top in $\WordPoset(m,n-1)$ which can thus be extended to a chain of length $(m+1)(n-1)-1+m+1 = (m+1)n-1$ from bottom to top in $\WordPoset(m,n)$.
\end{proof}

We conclude with the proof of the remaining part of our main theorem.

\begin{theorem}\label{thm:mn_extremal}
    For $m,n\geq 0$, the lattice $\WordPoset(m,n)$ is extremal. 
\end{theorem}
\begin{proof}
    If $n=0$, then $\WordPoset(m,0)$ is the singleton lattice, which is trivially extremal.

    So let $n>0$. By Corollary~\ref{cor:mn_semidistributive}, $\WordPoset(m,n)$ is semidistributive, which implies that $\bigl\lvert\JI\bigl(\WordPoset(m,n)\bigr)\bigr\rvert=\bigl\lvert\MI\bigl(\WordPoset(m,n)\bigr)\bigr\rvert$ by Lemma~\ref{lem:semidistributive_irreducibles}. Moreover, Lemma~\ref{lem:mn_irreducibles} implies that 
    \begin{displaymath}
        \bigl\lvert\JI\bigl(\WordPoset(m,n)\bigr)\bigr\rvert=(m+1)n-1.
    \end{displaymath}

    Lemma~\ref{lem:mn_long_chain} implies that $\ell\bigl(\WordPoset(m,n)\bigr)\geq (m+1)n-1$, which yields
    \begin{displaymath}
        (m+1)n-1 = \bigl\lvert\JI(\WordPoset(m,n)\bigr)\bigr\rvert \geq \ell\bigl(\WordPoset(m,n)\bigr)\geq (m+1)n-1, 
    \end{displaymath}
    and thus proves the theorem.
\end{proof}

\begin{corollary}\label{cor:mn_trim}
    For $m,n\geq 0$, the lattice $\WordPoset(m,n)$ is trim.
\end{corollary}
\begin{proof}
    This follows from Corollary~\ref{cor:mn_semidistributive} and Theorems~\ref{thm:semidistributive_extremal_trim} and \ref{thm:mn_extremal}.
\end{proof}

We finish this section by proving Theorem~\ref{thm:main_theorem}.

\begin{proof}[Proof of Theorem~\ref{thm:main_theorem}]
    This follows from Theorems~\ref{thm:mn_extremal} and \ref{thm:mn_doubling}.
\end{proof}

\section{Combinatorial Properties of $\WordPoset(m,n)$}
    \label{sec:mn_combinatorics}

\subsection{Cardinality}

A \defn{topless $(m,n)$-word} is an $(m,n)$-word that does not contain the letter $m+1$. By \eqref{it:mn_2}, every topless $(m,n)$-word is a weakly decreasing sequence of letters from $\{0,1,\ldots,m\}$; thus the number of topless $(m,n)$-words is $\binom{m+n}{n}$.

\begin{proposition}\label{prop:mn_words_cardinality}
    For $m,n\geq 0$, the cardinality of $\Words(m,n)$ is given by
    \begin{equation}\label{eq:mn_cardinality}
        \sum_{k=1}^{n}\binom{m+k}{k}\binom{n-1}{k-1}.
    \end{equation}
\end{proposition}
\begin{proof}
    Every $(m,n)$-word $\wb$ can be written as $w_1 a^{(1)} w_2 a^{(2)}w_3 a^{(3)}\cdots a^{(k-1)}w_k a^{(k)}$ for some $k\in[n]$, where each $a^{(i)}$ is a possibly empty sequence of $m+1$'s and the sum of the lengths of all $a^{(i)}$'s is $n-k$. We can thus view $(a^{(1)},a^{(2)},\ldots,a^{(k)})$ as a weak composition of $n-k$ into exactly $k$ parts. Thus, the number of all $(m,n)$-words is
    \begin{displaymath}
        \sum_{k=1}^{n}\binom{m+k}{k}\binom{n-1}{k-1}
    \end{displaymath}
    as desired.
\end{proof}

\subsection{Canonical Join Representations}

\begin{lemma}
    Let $\ub,\vb\in\Words(m,n)$. The join of $\ub$ and $\vb$ in $\WordPoset(m,n)$ is obtained by taking the componentwise maximum of $\ub$ and $\vb$.
\end{lemma}
\begin{proof}
    Let $\ub=u_1u_2\ldots u_n$ and $\vb=v_1v_2\ldots v_n$, and let $w_i=\max\{u_i,v_i\}$. By \eqref{it:mn_1}, we have $u_1\leq m$ and $v_1\leq m$, which implies $w_1\leq m$. Moreover, let $s\in[m]$ such that $w_i=s$. Then, without loss of generality we may assume that $u_i=s$, and it follows from \eqref{it:mn_2} that $u_j\geq s$ for all $j<i$. Consequently, $w_j=\max\{u_j,v_j\}\geq u_j\geq s$ for all $j<i$. It follows that $\wb=w_1w_2\ldots w_n\in\Words(m,n)$. 
    
    Moreover, it follows from the definition of the $w_i$ that $\ub\comporder\wb$ and $\vb\comporder\wb$ and that $\wb\comporder\wb'$ for all $\wb'\in\Words(m,n)$ with $\ub\comporder\wb'$ and $\vb\comporder\wb'$. This proves the claim.
\end{proof}

For $\vb\in\Words(m,n)$, we define
\begin{align*}
    \indeg(\vb) & \defs \bigl\lvert\{\ub\in\Words(m,n)\colon \ub\compcover\vb\}\bigr\rvert,\\
    \tops(\vb) & \defs \bigl\lvert\{i\in[2,n]\colon v_i=m+1\}\bigr\rvert.
\end{align*}

The \defn{support} of $\vb$ is the set $\Supp(\vb)\defs\{v_j\colon j\in[n]\;\text{and}\;v_j\in[m]\}$, \ie the set of all letters appearing in $\vb$ that are different from $0$ and $m+1$. 

\begin{lemma}\label{lem:mn_indegree}
    For $\vb=v_1v_2\ldots v_n\in\Words(m,n)$, we have 
    \begin{displaymath}
        \indeg(\vb)=\tops(\vb) + \bigl\lvert\Supp(\vb)\bigr\rvert.
    \end{displaymath}
\end{lemma}
\begin{proof}
    Suppose that there exists $i\in[n]$ such that $v_i=m+1$. Then, by \eqref{it:mn_1}, $i>1$. Let $s=\max\{v_j\colon 1\leq j<i\;\text{and}\;v_j<m+1\}$, and consider the word $\ub=u_1u_2\ldots u_n$ with $u_i=s$ and $u_j=v_j$ for $j\neq i$. Then, by design $u_i=s\geq u_j=v_j$ for all $j<i$ so that $\ub\in\Words(m,n)$ and $\ub\comporder\vb$. If there was $\ub'=u'_1u'_2\ldots u'_n\in\Words(m,n)$ with $\ub\compless\ub'\compless\vb$ then, necessarily $u_j=u'_j=v_j$ for all $j\neq i$ and $u_i<u'_i<v_j$. By design, however, there exists some $k<j$ such that $s=u_k=u'_k$ so that $u'_i>u_i=s=u'_k$ which contradicts \eqref{it:mn_2}, and we conclude $\ub\compcover\vb$.

    Now pick $s\in\Supp(\vb)$. Let $i$ be the maximum index such that $v_i=s$, and consider the word $\ub=u_1u_2\ldots u_n$ with $u_i=s-1$ and $u_j=v_j$ for all $j\neq i$. If $\ub\in\Words(m,n)$, then it is immediate that $\ub\compcover\vb$.  But since $\vb\in\Words(m,n)$, then by \eqref{it:mn_2} it follows that $v_j\geq v_i=s$ for all $j<i$ and therefore $u_j=v_j\geq s>u_i$ for all $j<i$. 

    Conversely, if $i$ is a non-maximum index such that $v_i=s$, \ie there exists $k>i$ such that $v_k=s$. Then, if we consider the word $\ub=u_1u_2\ldots u_n$ with $u_i<v_i$ and $u_j=v_j$, then $u_i<u_k$ but $i<k$ contradicting \eqref{it:mn_2}. 

    Therefore, every entry in the support of $\vb$ and every entry in $\vb$ equal to $m+1$ contribute an element covered by $\vb$. Since every element covered by $\vb$ is obtained by reducing some letter of $\vb$, this concludes the proof.
\end{proof}

\begin{corollary}\label{cor:refined_enumeration}
    The number of $\wb\in\Words(m,n)$ with $\indeg(\wb)=a$ and $\lvert\Supp(\wb)\rvert=b$ is 
    \begin{displaymath}
        \binom{m}{b}\binom{n-a+b}{b}\binom{n-1}{n-a+b-1}.
    \end{displaymath}
\end{corollary}
\begin{proof}
    If $\lvert\Supp(\wb)\rvert=b$, then $\wb$ contains $a-b$ letters equal to $m+1$. This means that the topless $(m,n)$-word obtained from $\wb$ by deleting all letters equal to $m+1$ has $n-a+b$ letters and comprises a weakly decreasing sequence with $b$ different letters and possibly some zeroes at the end. In other words, it can be regarded as a (descending) staircase of total length $n-a+b$ and steps (of length $\geq 1$) at heights $s\in\Supp(\wb)\cup\{0\}$. The lengths of these steps form a composition of $n-a+b$ into $b+1$ parts (if $\wb$ has a letter equal to $0$) or into $b$ parts (if $\wb$ has no letter equal to $0$). Therefore, for a given support set of size $b$, we can form $\binom{n-a+b-1}{b}+\binom{n-a+b-1}{b-1}=\binom{n-a+b}{b}$ of these staircases, and can augment each of them by inserting $a-b$ letters equal to $m+1$ as described in the proof of Proposition~\ref{prop:mn_words_cardinality}. Since we can choose $\binom{m}{b}$ such support sets, we get the formula from the statement.
\end{proof}

\begin{corollary}\label{cor:in_degree_enumeration}
    The number of $\wb\in\Words(m,n)$ with $\indeg(\wb)=a$ is
    \begin{displaymath}
        \sum_{b=0}^{a}\binom{m}{b}\binom{n-a+b}{n-a}\binom{n-1}{a-b}.
    \end{displaymath}
\end{corollary}

Computer experiments suggest the following formula equivalent to the sum in Corollary~\ref{cor:in_degree_enumeration}.

\begin{conjecture}
    The number of $\wb\in\Words(m,n)$ with $\indeg(\wb)=a$ is
    \begin{displaymath}
        \binom{m+a}{a}\binom{n}{a} - \binom{m+a-1}{m}\binom{n-1}{a-1}.
    \end{displaymath}
\end{conjecture}

\begin{proposition}\label{prop:mn_join_representation}
    The canonical join representation of $\wb=w_1w_2\ldots w_n\in\Words(m,n)$ is 
    \begin{equation}\label{eq:mn_canonical_join_rep}
        \Can(\wb) \defs \Bigl\{\ab^{(i,j)}\colon i=\max\{\ell\colon w_\ell=j\in\Supp(\wb)\}\Bigr\} \uplus \Bigl\{\bb^{(i)}\colon w_i=m+1\Bigr\}.
    \end{equation}
\end{proposition}
\begin{proof}
    As described in the proof of Lemma~\ref{lem:mn_indegree}, the positions $i\in[n]$ with $w_i=m+1$ or $i=\max\{\ell\colon w_\ell=s\in\Supp(\wb)\bigr\}$ each correspond to an element covered by $\wb$, so that the set in \eqref{eq:mn_canonical_join_rep} has the correct cardinality. Moreover, it is quickly verified that the join over \eqref{eq:mn_canonical_join_rep} is indeed $\wb$. 

    Assume that there exists some join-irreducible element $\ub'\in\Words(m,n)$ such that $\ub'\compless\ub$ for some $\ub\in\Can(\wb)$ and 
    \begin{equation}\label{eq:join_rep_replacement}
        \wb=\bigvee\bigl(\Can(\wb)\setminus\{\ub\}\bigr)\cup\{\ub'\}.
    \end{equation}
    Since each $\bb^{(i)}$ is an atom, we conclude that $\ub=\ab^{(i,j)}$ for some $i,j$. Then, the $i\th$ letter of $\ub'$ is strictly smaller than $j$. Since $w_i=j$, and no other word in $\Can(\wb)$ has $j$ as its $i\th$ letter, we obtain a contradiction to \eqref{eq:join_rep_replacement}. This concludes the proof.
\end{proof}

\begin{example}
    Let $\wb=474337720\in\Words(6,9)$. We have $w_2=w_6=w_7=7$ and $\Supp(\wb)=\{2,3,4\}$, where $\max\{\ell\colon w_\ell=4\}=3$, $\max\{\ell\colon w_\ell=3\}=5$ and $\max\{\ell\colon w_\ell=2\}=8$. Thus, the canonical join representation of $\wb$ is 
    \begin{multline*}
        \Bigl\{\ab^{(8,2)}, \ab^{(5,3)}, \ab^{(3,4)}, \bb^{(2)}, \bb^{(6)}, \bb^{(7)}\Bigr\}\\
            = \Bigl\{222222220, 333330000, 444000000, 070000000, 000007000, 000000700\Bigr\}.
    \end{multline*}
\end{example}

\begin{lemma}\label{lem:mn_irreducible_poset}
    For $m,n\geq 0$, the induced subposet of join-irreducibles of $\WordPoset(m,n)$ is isomorphic to 
    \begin{displaymath}
        (\mathbf{m}\times\mathbf{n})\uplus\underbrace{\mathbf{1}\uplus\mathbf{1}\uplus\cdots\uplus\mathbf{1}}_{n-1\;\text{times}}.
    \end{displaymath}
\end{lemma}
\begin{proof}
    By inspection, we see that $\{\bb^{(2)},\bb^{(3)},\ldots,\bb^{(n)}\}$ is an antichain of atoms of $\WordPoset(m,n)$, and $\bb^{(k)}\not\comporder\ab^{(i,j)}$ for all $i,j,k$. 

    Moreover, we have $\ab^{(i,j)}\comporder\ab^{(k,j)}$ if and only if $i\leq k$ and $\ab^{(i,j)}\comporder\ab^{(i,k)}$ if and only if $j\leq k$. Thus, the set $\bigl\{\ab^{(i,j)}\colon i\in[n], j\in[m]\bigr\}$ under $\comporder$ is isomorphic to the direct product of an $n$-chain and an $m$-chain.
\end{proof}

\subsection{The Chapoton-style $H$-Triangle}

Let us define 
\begin{displaymath}
    \pos(\wb) \defs \Can(\wb) \cap \bigl\{\ab^{(1,1)},\bb^{(2)},\bb^{(3)},\ldots,\bb^{(n)}\bigr\}
\end{displaymath}
to count the number of atoms of $\WordPoset(m,n)$ contained in the canonical join representation of $\wb$. Then, the \defn{$H$-triangle} of $\WordPoset(m,n)$ is
\begin{displaymath}
    \htriangle_{m,n}(x,y) \defs \sum_{\wb\in\Words(m,n)}x^{\indeg(\wb)}y^{\pos(\wb)}.
\end{displaymath}

\begin{remark}
    Our $H$-triangle is modeled after the prototypical polynomial introduced by F.~Chapoton in the context of cluster algebras and root systems~\cite{chapoton06sur}. 
\end{remark}

\begin{proposition}
    For $m,n\geq 0$ and $0\leq b\leq a\leq n$, the coefficient of $x^{a}y^{b}$ in $\htriangle_{m,n}(x,y)$ is equal to
    \begin{displaymath}
        h_{m,n,a,b} \defs \begin{cases}
            \binom{m}{a-b}\binom{n-b}{a-b}\binom{n-1}{b}, & \text{if}\;b<a-1,\\
            (mn-ma+m-1)\binom{n-1}{a-1}, & \text{if}\;b=a-1,\\
            \binom{n}{a}, & \text{if}\;b=a.
        \end{cases}
    \end{displaymath}
\end{proposition}
\begin{proof}
    As an auxiliary tool, we consider the following counting function, whose coefficients can be deduced from Corollary~\ref{cor:refined_enumeration}:
    \begin{align*}
        \tilde{\htriangle}_{m,n}(x,y) & \defs \sum_{\wb\in\WordPoset(m,n)}x^{\indeg(\wb)}y^{\tops(\wb)}\\
        & = \sum_{a=0}^{n}\sum_{b=0}^{a}\binom{m}{a-b}\binom{n-b}{a-b}\binom{n-1}{b} x^a y^b.
    \end{align*}
    Let $\wb\in\WordPoset(m,n)$. 
    \begin{itemize}
        \item If $\ab^{(1,1)}\notin\Can(\wb)$, then $\tops(\wb)=\pos(\wb)$. 
        \item If $\ab^{(1,1)}\in\Can(\wb)$, then Lemma~\ref{lem:mn_irreducible_poset} implies that $\ab^{(i,j)}\notin\Can(\wb)$ for $(i,j)\neq(1,1)$ as $\Can(\wb)$ is an antichain and $\ab^{(1,1)}\comporder\ab^{(i,j)}$. Thus, in this situation, Proposition~\ref{prop:mn_join_representation} states that $\wb=1\ub$, where $\ub$ is a word consisting only of the letters $0$ or $m+1$. It follows that $\indeg(\wb)=\pos(\wb)=\tops(\wb)+1$.
    \end{itemize}
    
    As the number of $\wb\in\WordPoset(m,n)$ with $\indeg(\wb)=a$ and $\ab^{(1,1)}\in\Can(\wb)$ is $\binom{n-1}{a-1}$, we obtain the following in combination with Corollary~\ref{cor:refined_enumeration}.
    \begin{itemize}
        \item The number of $\wb\in\Words(m,n)$ with $\indeg(\wb)=a$ and $\pos(\wb)=b<a-1$ equals the number of $\wb\in\Words(m,n)$ with $\indeg(\wb)=a$ and $\tops(\wb)=b$; this is $\binom{m}{a-b}\binom{n-b}{a-b}\binom{n-1}{b}$.
        \item The number of $\wb\in\Words(m,n)$ with $\indeg(\wb)=a$ and $\pos(\wb)=a-1$ equals the number of $\wb\in\Words(m,n)$ with $\indeg(\wb)=a$ and $\tops(\wb)=a-1$ minus the number of $\wb'\in\Words(m,n)$ with $\ab^{(1,1)}\in\Can(\wb')$; this is $\binom{m}{1}\binom{n-a+1}{1}\binom{n-1}{a-1}-\binom{n-1}{a-1}$.
        \item The number of $\wb\in\Words(m,n)$ with $\indeg(\wb)=a$ and $\pos(\wb)=a$ equals the number of $\wb\in\Words(m,n)$ with $\indeg(\wb)=a$ and $\tops(\wb)=a$ plus the number of $\wb'\in\Words(m,n)$ with $\ab^{(1,1)}\in\Can(\wb')$; this is $\binom{m}{0}\binom{n-a}{0}\binom{n-1}{a}+\binom{n-1}{a-1}$.
    \end{itemize}
    Some simplification then yields the claim.
\end{proof}

\subsection{The Galois Graph of $\WordPoset(m,n)$}

It is well known among lattice theorists that a finite lattice $\Poset$ can be represented as a set system ordered by inclusion. In such a representation, the sets corresponding to the lattice elements can be described by a binary relation, connecting the join- and the meet-irreducible elements of $\Poset$. This binary relation is the \defn{poset of irreducibles}~\cite{markowsky75factorization} or---essentially equivalent---the formal context~\cite{wille82restructuring} associated with $\Poset$. 

If $\Poset$ is extremal of length $n$, then the poset of irreducibles can essentially be described by a certain directed graph on the ground set $[n]$, the \defn{Galois graph} of $\Poset$. We refer the reader to \cite{thomas19rowmotion} for more background on and some examples of Galois graphs.  

If, moreover, $\Poset$ is interval constructable, then its Galois graph is the directed graph $\bigl(\JI(\Poset),\to\bigr)$, where the relation $\to$ is characterized as follows.

\begin{lemma}[{\cite[Corollary~A.18(ii)]{muehle21noncrossing}}]\label{lem:galois_characterization}
    Let $\Poset$ be an extremal, interval-constructable lattice. For $j,j'\in\JI(\Poset)$, we have $j\to j'$ if and only if $j\neq j'$ and $j'\leq {j'}^*\vee j$.
\end{lemma}

We use Lemma~\ref{lem:galois_characterization} to describe the Galois graph of $\WordPoset(m,n)$. Figure~\ref{fig:23_galois} shows the Galois graph of $\WordPoset(2,3)$.

\begin{theorem}\label{thm:mn_galois}
    Let $m,n\geq 0$. The Galois graph of $\WordPoset(m,n)$ is the directed graph $\Bigl(\JI\bigl(\WordPoset(m,n)\bigr),\to\Bigr)$, where for $\jb,\jb'\in\JI\bigl(\WordPoset(m,n)\bigr)$ with $\jb\neq\jb'$ we have
    \begin{displaymath}
        \jb\to\jb'\quad\text{if and only if}\quad
        \begin{cases}
            \jb=\ab^{(s,t)},\jb'=\ab^{(s',t')}\;\text{and}\;s\geq s',t\geq t',\\
            \jb=\bb^{(s)},\jb'=\ab^{(s',t')}\;\text{and}\;s=s'.
        \end{cases}
    \end{displaymath}
\end{theorem}
\begin{proof}
    Recall from Section~\ref{sec:extremality} that the join-irreducible elements of $\WordPoset(m,n)$ are of the following form:
    \begin{align*}
        \ab^{(s,t)} & = \underbrace{tt\ldots t}_{s\;\text{times}}00\ldots 0, && 1\leq s\leq n, 1\leq t\leq m\\
        \bb^{(s)} & = 00\ldots 0\underset{\overset{\uparrow}{s\th\;\text{letter}}}{(m\!+\!1)}0\ldots 0, && 2\leq s\leq n.
    \end{align*}

    Let $\jb,\jb'\in\JI\bigl(\WordPoset(m,n)\bigr)$. We thus need to check the condition from Lemma~\ref{lem:galois_characterization} in the following cases:
    
    (i) Suppose that $\jb=\ab^{(s,t)}$, $\jb'=\ab^{(s',t')}$. By definition, the ${s'}\th$ letter of ${\jb'}^{*}$ is strictly smaller than $t'$. Thus, $\ab^{(s,t)}\to\ab^{(s',t')}$ if and only if $s\geq s'$ and $t\geq t'$. 

    (ii) Suppose that $\jb=\ab^{(s,t)}$, $\jb'=\bb^{(s')}$.  In that case, we have ${\jb'}^{*}=\ob$, and thus $\bb^{(s')} = \jb' \leq {\jb'}^{*}\vee\jb = \ob\vee\jb = \jb = \ab^{(s,t)}$. As no letter of $\ab^{(s,t)}$ is equal to $m+1$, we get $\ab^{(s,t)}\not\to\bb^{(s')}$.

    (iii) Suppose that $\jb=\bb^{(s)}$, $\jb'=\ab^{(s',t')}$. By definition, the ${s'}\th$ letter of ${\jb'}^{*}$ is strictly smaller than $t'$. Thus, $\jb'\leq {\jb'}^{*}\vee\jb$ can only be true if the ${s'}\th$ letter of $\jb$ is at least $t'$. However, if $s'=t'=1$, then ${\jb'}^{*}=\ob$, and $\jb'\leq\jb$ will never be true. 
    
    By definition, $s>1$ and the only non-zero letter of $\bb^{(s)}$ is the $s\th$ letter which is equal to $m+1$, which implies that $\bb^{(s)}\to\ab^{(s',t')}$ if and only if $s=s'$.

    (iv) Suppose that $\jb=\bb^{(s)}$, $\jb'=\bb^{(s')}$. In that case, we have ${\jb'}^{*}=\ob$, and thus $\jb'\leq {\jb'}^{*}\vee\jb = \ob\vee\jb = \jb$ if and only if $\jb'=\jb$ as both $\bb^{(s)}$ and $\bb^{(s')}$ are atoms. Thus, we have $\bb^{(s)}\not\to\bb^{(s')}$ for all $s,s'$.
\end{proof}

\begin{remark}
    For $m=1$, Theorem~\ref{thm:mn_galois} is equivalent to \cite[Theorem~1.1]{muehle22hochschild}. 
\end{remark}

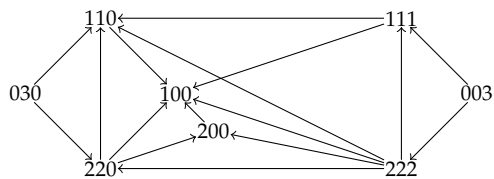
\begin{figure}
    \centering
    \begin{tikzpicture}
        \def\x{1};
        \def\y{1};
        \draw(1*\x,4*\y) node(b2)[wordnode]{$030$};
        \draw(7*\x,4*\y) node(b3)[wordnode]{$003$};
        \draw(3*\x,4*\y) node(a11)[wordnode]{$100$};
        \draw(2*\x,5*\y) node(a21)[wordnode]{$110$};
        \draw(6*\x,5*\y) node(a31)[wordnode]{$111$};
        \draw(3.5*\x,3.5*\y) node(a12)[wordnode]{$200$};
        \draw(2*\x,3*\y) node(a22)[wordnode]{$220$};
        \draw(6*\x,3*\y) node(a32)[wordnode]{$222$};
        \draw[->](b2) -- (a21);
        \draw[->](b2) -- (a22);
        \draw[->](b3) -- (a31);
        \draw[->](b3) -- (a32);
        \draw[->](a32) -- (a31);
        \draw[->](a22) -- (a21);
        \draw[->](a32) -- (a22);
        \draw[->](a32) -- (a11);
        \draw[->](a32) -- (a21);
        \draw[->](a31) -- (a21);
        \draw[->](a31) -- (a11);
        \draw[->](a12) -- (a11);
        \draw[->](a21) -- (a11);
        \draw[->](a22) -- (a11);
        \draw[->](a22) -- (a12);
        \draw[->](a32) -- (a12);
    \end{tikzpicture}
    \caption{The Galois graph of $\WordPoset(2,3)$.}
    \label{fig:23_galois}
\end{figure}


\end{document}